\documentclass[11pt]{article}
\usepackage[dvips]{epsfig}
\usepackage{amsmath,amssymb,euscript,graphicx,amsfonts,verbatim}
\usepackage{enumerate,color}
\usepackage{graphicx}
\let\bowtie\relax
\DeclareFontEncoding{LS1}{}{}
\DeclareFontSubstitution{LS1}{stix}{m}{n}
\DeclareSymbolFont{STIXsymbols}{LS1}{stixscr}{m}{n}
\DeclareMathSymbol{\bowtie}{\mathrel}{STIXsymbols}{"0E}

\usepackage{graphicx}

\usepackage[colorlinks=true,linkcolor=blue,urlcolor=blue,citecolor=blue]{hyperref}
\usepackage{float}
\usepackage{nicematrix}

\newtheorem{thm}{Theorem}[section]
\newtheorem{prop}[thm]{Proposition}
\newtheorem{lem}[thm]{Lemma}

\newtheorem{conj}[thm]{Conjecture}

\newtheorem{rmk}[thm]{Remark}

\DeclareMathOperator\Aut{Aut}

\DeclareMathOperator\sym{Sym}
\DeclareMathOperator\alt{Alt}

\DeclareMathOperator\Der{Der}

\DeclareMathOperator{\cay}{Cay}
\DeclareMathOperator{\Fix}{F}
\DeclareMathOperator{\soc}{Soc}

\newcommand{\agl}[2]{\operatorname{AGL}_#1(#2)}
\newcommand{\pgl}[2]{\operatorname{PGL}_#1(#2)}
\newcommand{\psl}[2]{\operatorname{PSL}_#1(#2)}

\newcommand{\mathieu}[1]{\operatorname{M}_{#1}}

\newcommand{\sln}[2]{\operatorname{SL}_#1(#2)}
\newcommand{\congr}[2]{=\ #1\ (\operatorname{mod}\ #2)}
\newcommand{\PAut}{\operatorname{PAut}}
\newcommand{\po}[2]{\operatorname{P\Omega}_{#1}^\varepsilon(#2)}
\newcommand{\pg}[2]{\operatorname{PG}_{#1}(#2)}

\newcommand{\Ker}{\operatorname{Ker}}

\newcommand{\Qed}{\rule{2.5mm}{3mm}}

\newcounter{case}

\renewcommand{\thecase}{\arabic{case}}

\newcounter{subcase}

\numberwithin{subcase}{case}

\newenvironment{proof}{{\noindent \sc Proof.}}{\hfill $\Qed$ \\}

\begin{document}


\begin{center}
{\bf\Large Intersection density of imprimitive groups of degree $pq$} \\ [+4ex]  
\large Angelot Behajaina{\textsuperscript{a}}, 
\addtocounter{footnote}{0} 
\large Roghayeh Maleki{\textsuperscript{b,c}} and  
\addtocounter{footnote}{0} 

Andriaherimanana Sarobidy Razafimahatratra{\textsuperscript{b,c,}\footnote{ corresponding author: sarobidy@phystech.edu}} 
\\ [+2ex]
{\it \small 
	\textsuperscript{a}Department of Mathematics, Technion - Israel Institute of Technology, Haifa, Israel\\
\textsuperscript{b}University of Primorska, UP FAMNIT, Glagolja\v ska 8, 6000 Koper, Slovenia\\
\textsuperscript{c}University of Primorska, UP IAM, Muzejski trg 2, 6000 Koper, Slovenia\\
}
\end{center}

\begin{abstract}

	A subset $\mathcal{F}$ of a finite transitive group $G\leq \sym(\Omega)$ is \emph{intersecting} if any two elements of $\mathcal{F}$ agree on an element of $\Omega$. The \emph{intersection density} of $G$ is the number $$\rho(G) = \max\left\{ \mathcal{|F|}/|G_\omega| \mid \mathcal{F}\subset G \mbox{ is intersecting} \right\},$$ where $\omega \in\Omega$ and $G_\omega$ is the stabilizer of $\omega$ in $G$.
	It is known that if $G\leq \sym(\Omega)$ is an imprimitive group of degree a product of two odd primes $p>q$ admitting a block of size $p$ or two complete block systems, whose blocks are of size $q$, then $\rho(G) = 1$.
	
	In this paper, we analyse the intersection density of imprimitive groups of degree $pq$ with a unique block system with blocks of size $q$ based on the kernel of the induced action on blocks. For those whose kernels are non-trivial, it is proved that the intersection density is larger than $1$ whenever there exists a cyclic code $C$ with parameters $[p,k]_q$ such that any codeword of $C$ has weight at most $p-1$, and under some additional conditions on the cyclic code, it is a proper rational number. For those that are quasiprimitive, we reduce the cases to almost simple groups containing $\alt(5)$ or a projective special linear group. We give some examples where the latter has intersection density equal to $1$, under some restrictions on $p$ and $q$.
\end{abstract}

\begin{quotation}
\end{quotation}


\section{Introduction}

All groups considered in this paper are finite. Given a finite set $\Omega$, we let $\sym(\Omega)$ be the symmetric group of $\Omega$. A set $\mathcal{F}\subset \sym(\Omega)$ is \emph{intersecting} if for any $g,h\in \mathcal{F}$, there exists $\omega\in \Omega$ such that $\omega^g = \omega^h$. We are interested in studying the largest intersecting sets in finite permutation groups. Problems of this nature have been studied since the late 70s with the paper of Deza and Frankl \cite{Frankl1977maximum} on the size and the possible structures of the largest intersecting sets in the symmetric group $\sym(n)$ with its natural action on $\{1,2,\ldots,n\}$. Deza and Frankl proved that if $\mathcal{F} \subset \sym(n)$ is intersecting, then $|\mathcal{F}|\leq (n-1)!$. Later, Cameron and Ku \cite{cameron2003intersecting}, independently, Larose and Malvenuto \cite{larose2004stable}, proved that equality holds if and only if $\mathcal{F}$ is a coset of a stabilizer of a point of $\sym(n)$. Consequently, the maximum intersecting sets in $\sym(n)$ are cosets of a stabilizer of a point. 

A natural question is then whether similar results to the one on the symmetric group hold for general transitive permutation groups. Unfortunately, the answer to this question is negative. To see this, consider the alternating group $G = \alt(4)$ acting on left cosets of $H = \langle (1\ 2)\rangle$ by left multiplication. This action is transitive and its point stabilizers have size $2$, however, the Klein group $K = \langle (1\ 2),(3\ 4) \rangle$ is an intersecting set of size $4$. 

In \cite{li2020erd}, a measure for intersecting sets in transitive groups was introduced. Formally, the \emph{intersection density} of a transitive group $G\leq\sym(\Omega)$ is the rational number
\begin{align*}
	\rho(G) := \max \left\{ \rho(\mathcal{F}) \mid \mathcal{F} \subset G \mbox{ is intersecting} \right\},
\end{align*}
with $\rho(\mathcal{F}) := |\mathcal{F}|/|G_\omega|$ for any intersecting set $\mathcal{F}$, where $G_\omega$ is the point stabilizer of $\omega \in \Omega$ in $G$. The study of the intersection density parameter in transitive groups has recently drawn the attention of many researchers \cite{behajaina20203,hujdurovic2022intersection,2021arXiv210803943H,hujdurovic2021intersection,meagher2019erdHos,meagher2021erdHos,meagher180triangles,meagher2016erdHos,gl2,AMC2554,spiga2019erdHos}. One of the major results about the intersection density of a group was proved in \cite{meagher180triangles}. It was proved that the largest possible value of the intersection density of a group acting transitively on $\Omega$, with $|\Omega| \geq 3$, is $\frac{|\Omega|}{3}$ and this upper bound is sharp since  it is attained by the groups given in \cite[Theorem~5.1]{meagher180triangles}.

{The study of the intersection density of the transitive group $G\leq \sym(\Omega)$, as we will see later in Section~\ref{sect:background}, is equivalent to the study of the cocliques of a certain Cayley graph $\Gamma_G$ called the derangement graph of $G$. The clique-coclique \cite{godsil2016erdos} bound states that \footnote{$\alpha(\Gamma_G)$ and $\omega(\Gamma_G)$ denote respectively the maximum size of cocliques and cliques in $\Gamma_G$} $\alpha(\Gamma_G)\omega(\Gamma_G) \leq |G|$. The knowledge of cocliques in $\Gamma_G$ can therefore be used to give an upper bound on certain complex structures on $G$ arising from coding theory, provided that $G$ is highly transitive. First recall that $\mathcal{F} \subset G$ is called \emph{$k$-intersecting} if there exist distinct $\omega_1,\omega_2,\ldots,\omega_k \in \Omega$ such that $\omega^g_i = \omega^h_i$ for any $g,h\in \mathcal{F}$ and $1\leq i\leq k$. Let $\Omega^{(k)}$ be the set of all $k$-tuples of $\Omega$ with pairwise distinct entries. It is not hard to see that a $k$-intersecting set of $G\leq \sym(\Omega)$ is exactly an intersecting set of $G\leq \sym(\Omega^{(k)})$, provided that $G$ is also $k$-transitive, that is, it is transitive on $\Omega^{(k)}$. For any two permutations $g,h\in G$, define $\operatorname{d}(g,h) = |\Omega|-\left|\left\{\omega\in \Omega\ :\ \omega^g = \omega^h \right\}\right|$ to be the distance between $g$ and $h$. A $\operatorname{PC}(G,k)$ \emph{permutation code} of a transitive group $G\leq \sym(\Omega)$ is a subset $S \subset G$ such that $\operatorname{d}(g,h)\geq k$, for any two elements $g,h\in G$. A $\operatorname{PC}(G,k)$ permutation code is in fact a clique in the derangement graph of $G\leq \sym\left(\Omega^{(|\Omega|-k+1)}\right)$. If $\Gamma_{G}^{(|\Omega|-k+1)}$ is the derangement graph of $G\leq \sym\left(\Omega^{(|\Omega|-k+1)}\right)$, then by using the clique-coclique bound on $\Gamma_G^{(|\Omega|-k+1)}$,  we obtain a general upper bound on the maximum size of a $\operatorname{PC}(G,k)$ permutation code. This yields an upper bound which is not tight in general \cite{bereg2020lower,dukes2010bounds} and the classification of multiply transitive groups forces $k$ to be large, unless $G$ is equal to the full symmetric group $\sym(\Omega)$ or the alternating group $\alt(\Omega)$. A more general approach to permutation codes arising from permutation groups is given in \cite{cameron2010permutation}.}

The overwhelming majority of the works on the intersection density of transitive groups focus on trying to classify the transitive groups whose intersection densities are equal to $1$ or whose largest intersecting sets are cosets of a point stabilizer. For instance, Meagher, Spiga and Tiep \cite{meagher2016erdHos} proved that any finite $2$-transitive group has intersection density $1$. In \cite{li2020erd,meagher180triangles}, works on transitive groups that have large intersection densities were initiated.

Another interesting problem is to study the intersection density of groups with fixed degrees.
It was conjectured in \cite[Conjecture~6.6]{meagher180triangles} that if $G$ is a transitive group of degree a product of two distinct primes $q<p$, then $\rho(G) = 1$ if $p$ and $q$ are odd, and $\rho(G) \leq 2$ if $q=2$. The latter  was proved in \cite{AMC2554}. In \cite{hujdurovic2021intersection},  Hujdurovi\'{c} et al. proved further that the intersection density of such groups is either $1$ or $2$; if it is equal to $2$, then the group is a split extension of an elementary abelian $2$-group. In \cite{hujdurovic2022intersection},  Hujdurovi\'{c} et al. gave an elegant construction of transitive groups disproving the conjecture when $p$ and $q$ are both odd. They provided a family of imprimitive groups of degree $pq$ ($q<p$ are odd primes) whose intersection densities are equal to $q$. 

Although \cite[Conjecture~6.6~(3)]{meagher180triangles} is false, it is natural to ask about the possible values of the intersection density of transitive groups of degree a product of two distinct odd primes $p>q$. When $G$ is primitive of degree $pq$, the third author proved in \cite{razafimahatratra2021intersection} that $\rho(G) = 1$ for several families. It is in fact conjectured in \cite[Conjecture~9.1]{razafimahatratra2021intersection} that all primitive groups of degree a product of two odd primes have intersection densities equal to $1$. Imprimitive groups of degree $pq$ that admit a block system with blocks of size $p$ have intersection density equal to $1$ (see \cite{hujdurovic2022intersection}). It was also proved in \cite{razafimahatratra2021intersection} that any imprimitive group of degree a product of two odd primes with at least two (imprimitivity) block systems has intersection density equal to $1$. 
The objective of this paper is to study the intersection density of the remaining families of imprimitive groups of degree a product of two primes $p$ and $q$ (where $p>q$), i.e., those with a unique block system of size $q$. To this end, for any $n\geq 2$, we define the set
\begin{align*}
	\mbox{$\mathcal{I}_{n}:= \{ \rho(G)  \mid G \mbox{ is transitive of degree }n \}$.}
\end{align*}
The set $\mathcal{I}_n$, for $n\geq 2$, was first introduced in \cite{meagher180triangles}. For any prime $p$, we have $\mathcal{I}_p = \{1\}$ since any transitive group of prime degree contains a regular subgroup, which implies that its intersection density is equal to $1$ (see Lemma~\ref{lem:no-hom}). For any odd prime $p$, it follows from \cite{hujdurovic2021intersection} that $\mathcal{I}_{2p} = \{ 1,2 \}$. When $p$ and $q$ are two distinct odd primes ($q<p$), it is known that any transitive group $G\leq \sym(\Omega)$ of degree $pq$ admits a semiregular element of order $p$ \cite{maruvsivc1981vertex}. Using this fact, one can prove that $\rho(G)\leq q$, for any transitive group of degree $pq$ (see Lemma~\ref{lem:no-hom}). This upper bound is sharp since the transitive groups given in \cite{hujdurovic2022intersection} attain it. Apart from these few results, very little is known about the set $\mathcal{I}_{pq}$.

Henceforth, we will assume that $p>q$ are odd primes and $G\leq \sym(\Omega)$ is imprimitive of degree $pq$ admitting a unique block system $\mathcal{B}$ with blocks of size $q$. Let $\overline{G} \leq \sym(\mathcal{B})$ be the induced action of $G$ on $\mathcal{B}$ and denote the kernel of the canonical epimorphism $G \to \overline{G}$ by $\Ker(G\to \overline{G})$. If $G$ is \emph{quasiprimitive}, i.e., all its non-trivial normal subgroups are transitive, then $\Ker(G \to \overline{G})$ is trivial since it is intransitive. Conversely, if $\Ker(G \to \overline{G})$ is trivial, then a non-trivial normal subgroup of $G$ is transitive, otherwise its orbits would form a block system of $G$ (see Lemma \ref{lem:normal-blocks}) and, by uniqueness of $\mathcal{B}$, this normal subgroup would be trivial. Hence, $G$ is quasiprimitive if and only if $\Ker(G\to \overline{G})$ is trivial. Consequently, we shall distinguish the two cases whether $\Ker(G\to \overline{G})$ is trivial or not. 

The first part of this paper focuses on the case where the kernel of the induced action is trivial (i.e., the quasiprimitive cases). We first reduce the quasiprimitive cases to groups containing one of two almost simple groups. Since this reduction is an easy exercise that follows from the classification of the transitive groups of prime degree, due to Burnside in 1900, we only state it as a remark.

\begin{rmk}
	If $G\leq \sym(\Omega)$ is an imprimitive quasiprimitive of degree a product of two odd primes $p>q$, then $G$ is an almost simple group with socle equal to 
	\begin{enumerate}[(i)]
		\item $\alt(5)$ admitting a transitive action of degree $3\times 5 $,\label{rmk:quasi-alt}
		\item $\psl{2}{11}$ admitting a transitive action of degree $5\times 11$, or \label{rmk:quasi-psl2}
		\item $\psl{d}{r^{d^m}}$ admitting a transitive action of degree $pq$, for some $d\geq 2, \ m\geq 0$, and a prime number $r$ such that $p = \frac{r^{{d^{m+1}}}-1}{r^{d^m}-1}$.\label{rmk:quasi-psl}
	\end{enumerate}
	\label{thm:quasiprimitive}
\end{rmk}
{A discussion regarding this result is given in Section~\ref{sect:quasiprimitive} for sake of completeness.}
The groups in \eqref{rmk:quasi-alt} and \eqref{rmk:quasi-psl2} in Remark~\ref{thm:quasiprimitive} have intersection densities equal to $1$ via \verb*|Sagemath| \cite{sagemath}. Hence, the remaining minimal cases for the quasiprimitive case are the imprimitive actions of $\psl{d}{r^{d^m}}$, of degree a product of two odd primes. These groups are extremely difficult to deal with and we believe that new techniques are required to determine their intersection density in general. We present a case where the intersection density of such groups can be computed using the known techniques in the literature. This result is a particular case of Remark~\ref{thm:quasiprimitive}\eqref{rmk:quasi-psl} where $m=0 \mbox{ and } d = r$, and in this case, there is a unique imprimitive permutation representation of degree $pq$ of $\psl{r}{r}$, where $p = \frac{r^r-1}{r-1}$ and $q = r$.

\begin{thm}
	For any odd prime $q$, if $p= \frac{q^q-1}{q-1}$ is a prime, then the intersection density of the imprimitive representation of $\psl{q}{q}$ of degree $pq$ is $1$.\label{thm:psl-case}
\end{thm}
Based on computational results, we make the following conjecture.
\begin{conj}
	If $G$ is a quasiprimitive group of degree $pq$, then $\rho(G) = 1$.\label{conj2}
\end{conj}

The remaining results of the paper are about the non-quasiprimitive cases, i.e., when $\Ker(G\to \overline{G})$ is non-trivial.
Our next result is stated as follows.
\begin{thm}
	Let $G\leq \sym(\Omega)$ be a transitive group of degree $pq$ admitting a unique block system $\mathcal{B}$ whose blocks are of size $q$. Suppose that $H = \Ker(G\to \overline{G})$ is non-trivial.  Let $N$ be a minimal normal subgroup of $G$ contained in $H$.  
	\begin{enumerate}[(i)]
		\item If $N$ is non-abelian, then $\rho(G) = 1$.\label{i}
		\item If $N$ is abelian and $H$ admits a derangement, then $\rho(G) = 1$.\label{ii}
	\end{enumerate}\label{thm:codes}
\end{thm}

The counterexamples to \cite[Conjecture~6.6~(3)]{meagher180triangles} given in \cite{hujdurovic2021intersection} are based on the existence of certain cyclic codes (see \S\ref{sec:conimcod} for the basic definitions). In this paper, we also prove a certain converse to their construction.

\begin{thm}
	Let $G$ be a group as in Theorem \ref{thm:codes}. Moreover, assume that $G$ is minimally transitive with intersection density $\rho(G) = q$. Then, there is a $[p,k]_q$ cyclic code $C$ such that the Hamming weight $w(\mathbf{c})$ is at most $p-1$ for all $\mathbf{c}\in C$.\label{thm:characterization}
\end{thm}

These two results imply that if $\Ker(G\to \overline{G})$ is non-trivial and $\rho(G)>1$, then $G$ is an extension of a certain derangement-free elementary abelian $q$-group. We conjecture the following based on computational results via \verb*|Sagemath| \cite{sagemath}.
\begin{conj}
	Assume that $\Ker(G\to \overline{G})$ is non-trivial. If $\overline{G}$ is $2$-transitive, then $\rho(G) = 1$. If $\overline{G} = C_p\rtimes C_k< \agl{1}{p}$, then $\rho(G) \in \left\{1\right\} \cup \left\{  \frac{q}{d} : d\mid k \mbox{ and } d\leq q \right\}$.\label{conj1}
\end{conj}
This conjecture is hard as far as we know.	
With additional restrictions on the odd primes $p$ and $q$, we prove that $\mathcal{I}_{pq}$ contains a proper rational number.
\begin{thm}
	If $q$ and $p = \frac{q^k-1}{q-1}$ are odd primes such that $k<q$, then $\frac{q}{k} \in \mathcal{I}_{pq}$.
	\label{thm:main}
\end{thm}

This paper is organized as follows. In Section~\ref{sect:background}, we  provide the necessary background to  prove the main  results. The proof of Theorem~\ref{thm:psl-case} and a discussion about Remark~\ref{thm:quasiprimitive} are given in Section~\ref{sect:quasiprimitive}. In Section~\ref{sect:code-density}, we prove Theorem~\ref{thm:codes} and Theorem~\ref{thm:characterization}. We give the proof of Theorem~\ref{thm:main} in Section~\ref{sect:rational}, and in Section~\ref{sect:future}, we give a conjecture (see Conjecture~\ref{conj:main}) that encompasses all the possibilities for the set $\mathcal{I}_{pq}$, for any values of $p$ and $q$.

\section{Background results}\label{sect:background}
In this section, we review some basics on transitive groups (\S\ref{ssec:transgro} and \S\ref{ssec:transpq}) and codes (\S\ref{ssec:codesq}), that are needed in this work. 
\subsection{Transitive groups}\label{ssec:transgro}
Throughout this section, we let $G\leq\sym(\Omega)$ be a finite transitive group. 
\subsubsection{Intersection density}
The standard technique to study the intersection density of a transitive group is to transform the problem into a graph theoretical one. Recall that if $K$ is a group and  $S \subset K\setminus \{1\}$ is a subset such that $x\in S$ implies $x^{-1} \in S$, then the Cayley graph $\cay(K,S)$ is the graph whose vertex set is $K$ and $x,y\in K$ are adjacent if $yx^{-1} \in S$. It is well known that $\cay(K,S)$ is connected if and only if $\langle S\rangle = K$. Furthermore, the number of components of $\cay(K,S)$ is equal to $[K:\langle S\rangle]$. Given a permutation group $K$, the \emph{derangement graph} of $K$, denoted by $\Gamma_K$, is the graph $\cay(K,\Der(K))$, where $\Der(K)= \{ k\in K \mid \mbox{ $k$ is a derangement} \}$. If $\Der(K) = \varnothing$, then $\Gamma_{K}$ is the empty graph, i.e., without edges. If $K$ is transitive of degree at least $2$, then  $\Der(K)\neq \varnothing$ by \cite{jordan1872recherches}, and so $\Gamma_K$ has at least one edge.

Recall that an \emph{independent set} or a \emph{coclique} in a graph is a subgraph in which no two vertices are adjacent. The size of the largest cocliques in a graph $X$ is called the \emph{independence number} and is denoted by $\alpha(X)$. It is not hard to see that given a transitive group $G \leq \sym(\Omega)$, a subset $\mathcal{F} \subset G$ is intersecting  if and only if $\mathcal{F}$ is a coclique in $\Gamma_G$. Consequently, we have $\rho(G) = \alpha(\Gamma_G)/|G_\omega|,$ for $\omega\in \Omega$.

We recall the following bound on the intersection density of a transitive group.
\begin{lem} \cite[Corollary~4.2]{razafimahatratra2021intersection}
	Let $H\leq \sym(\Omega)$ be a transitive group and $K$ be a subgroup of $H$ with $k$ orbits of the same size. Then, we have 
	$
	\rho(H) \leq \rho(K)k.
	$\label{lem:no-hom}
\end{lem}

\subsubsection{Imprimitive groups}
In this part, we further assume that $G\leq \sym(\Omega)$ is imprimitive. 
Recall that $S\subset \Omega$ is a \emph{block (of imprimitivity)} of $G$ if $S^g \cap S = \varnothing$ or $S^g = S$, for all $g\in G$. It is clear that $\{\omega\}$ for $\omega \in \Omega$ and the set $\Omega$ are always blocks of the group $G$; we call such blocks \emph{trivial}. If $G$ admits a block $S$ such that $1<|S|<|\Omega|$, then we say that the group $G$ is \emph{imprimitive}. In this case, we know the collection of subsets given by $\mathcal{B} = \{ S^g \mid g\in G \}$ forms a partition of $\Omega$, that is invariant by the action of $G$. Such a $G$-invariant partition of $\Omega$ is  called a \emph{(complete) block system}.
We recall the following result.
\begin{lem}[Theorem~1.6A \cite{dixon1996permutation}]
	If $N\trianglelefteq G$, then the orbits of $N$ form a $G$-invariant partition of $\Omega$. That is, $\{ \omega^N \mid \omega \in \Omega \}$ is a $G$-invariant partition of $\Omega$.\label{lem:normal-blocks}
\end{lem}
Note that in Lemma~\ref{lem:normal-blocks} if $N = 1$ or $N$ is transitive, then the corresponding block systems are $\{ \{\omega \} \mid \omega\ \in \Omega\}$ or $\{\Omega\}$, respectively. That is, they consist of trivial blocks.

Recall that a subgroup $K\leq G$ is \emph{semiregular} if its stabilizers are trivial. Moreover, an element $g \in G$ is called {\it semiregular} if $\langle g \rangle \leq G$ is semiregular. Therefore, it is clear that any derangement $g \in G$ of prime order is semiregular. Next, we recall the following lemma (see \cite[Lemma 3.1]{hujdurovic2021intersection}).
\begin{lem}
	Let $G$ be a transitive group admitting a semiregular subgroup $K$ whose orbits form a $G$-invariant partition $\mathcal{B}$. Let $\overline{G}$ be the induced action of $G$ on $\mathcal{B}$. Then, we have $\rho(G) \leq \rho(\overline{G})$. \label{lem:semi-regular}
\end{lem}

\subsubsection{The subgroup generated by non-derangements} In this section, we briefly describe the structure of the subgroup generated by the non-derangements in a transitive group. We assume that $G\leq \sym(\Omega)$ is a transitive group, where $\Omega$ is a finite set. Define
\begin{equation}\label{eq:fixg}
	\Fix(G) := \langle \bigcup_{\omega\in \Omega}G_\omega \rangle.
\end{equation}
In other words, $\Fix(G)$ is the subgroup generated by all permutations with at least one fixed point (i.e., the non-derangements). 
It is not hard  to see that $\Fix(G) \trianglelefteq G$ since the point stabilizers are conjugate. If $\Fix(G) = \{1\}$, then each point stabilizer of $G$ is trivial, thus $G$ is regular. If $\Fix(G) \neq \{1\}$, then its orbits form a block system (see Lemma~\ref{lem:normal-blocks}). 

If $\Fix(G)< G$, then the complement of the derangement graph $\Gamma_G$ is $\overline{\Gamma}_G = \cay(G,G\setminus \left( \Der(G) \cup \{1\}\right))$. As $\langle G\setminus \left( \Der(G) \cup \{1\} \right) \rangle = \Fix(G)$ is a proper subgroup of $G$, we can see that 
$\overline{\Gamma}_G$ is disconnected.  The number of components of $\overline{\Gamma}_G$ is equal to $[G:\Fix(G)]$ since $\overline{\Gamma}_G$ is also a Cayley graph. Switching to the complement, we conclude that $\Gamma_G$ is a \emph{join} (see \cite[pg~21]{Harary1969} or \cite{zykov1949some} for the definition) of $[G:\Fix(G)]$ copies of the complement of $\overline{\Gamma}_G$ that contains $1$, that is, the graph $\Gamma_{\Fix(G)} = \cay(\Fix(G),\Fix(G)\cap \Der(G)).$ 
From this explanation, the following proposition follows.
\begin{prop}
	The graph $\Gamma_G$ is complete multipartite if and only if $\Fix(G)$ is intersecting.
\end{prop}

If $\Fix(G) = G$, then the graph $\Gamma_G$ is not a join anymore.
We recall the following result, whose proof can be found in \cite{meagher180triangles}.
\begin{prop}\cite[Lemma~2.3]{meagher180triangles}
	A subgroup $H\leq G$ is intersecting if and only if it is derangement-free.
\end{prop}

\subsection{The transitive groups of degree $pq$}\label{ssec:transpq}
In this section, we categorize the types of transitive groups of degree a product of two odd primes for a better analysis of their intersection density. Let $G\leq \sym(\Omega)$ be a transitive group of degree a product of two odd primes $p>q$. Let $m$ be the number of block systems of $G$. In \cite{lucchini1991imprimitive}, Lucchini proved that if $m\geq 2$, then $m = 2$ or $m = p+1$ and $G$ has at most one block system with blocks of size $p$. It was proved in \cite{hujdurovic2022intersection} that if $G$ admits a block of size $p$, then $\rho(G) = 1$. If $m\geq 2$, then it was also shown in \cite{razafimahatratra2021intersection} that $\rho(G) = 1$. Therefore, we only need to consider the cases where $G$ admits only trivial blocks or a unique block system where the blocks are of size $q$. Consequently, $G$ is one of the following types.
\begin{enumerate}[(I)]
	\item \underline{{\bf Quasiprimitive types}:} every non-trivial normal subgroup of $G$ is transitive. These groups can be further subdivided into the following.
	\begin{enumerate}[(a)]
		\item {\bf Primitive types:} $G$ only admits trivial blocks. Some groups of this type has been proved to have intersection density $1$ in \cite{razafimahatratra2021intersection}.
		\item {\bf Quasiprimitive types admitting non-trivial blocks:} $G$ admits a block system whose blocks are of size $q$ and whose kernel of the induced action on blocks is trivial.
	\end{enumerate}
	\item \underline{{\bf Genuinely imprimitive types:}} $G$ admits a block system with blocks of size $q$ whose kernel of the induced action on the blocks is non-trivial.
\end{enumerate}

\subsection{Cyclic Codes}\label{ssec:codesq}
Let $n \geq 1$ be a positive integer and $q$ be a prime number. The {\it Hamming weight} of $\mathbf{c} \in \mathbb{F}_q^n$, denoted by $w(\mathbf{c})$, is the number of non-zero entries of $\mathbf{c}$. We also define $Z(\mathbf{c})=n-w(\mathbf{c})$. The standard non-degenerate symmetric bilinear form $\langle \cdot, \cdot \rangle$ on $\mathbb{F}_q^n$ is given by:
$$
\langle \mathbf{c}, \mathbf{d} \rangle=c_1d_1+\dots+c_nd_n,
$$ 
for all $\mathbf{c}=(c_1,\dots,c_n), \mathbf{d}=(d_1,\dots,d_n) \in \mathbb{F}_q^n$. 
The symmetric group on $\{0,\dots,n-1\}$, which is denoted by  $S_n$, acts on $\mathbb{F}_q^n$ as follows:
$$
\mathbf{c}^\gamma=(c_{0^\gamma},\dots,c_{(n-1)^\gamma}),
$$
for $\mathbf{c}=(c_0,\dots,c_{n-1}) \in\mathbb{F}_q^n$ and $\gamma \in S_n$.

Now let $C$ be an $[n,k]_q$ {\it linear code} $C$ over $\mathbb{F}_q$, that is, a $k$-dimensional $\mathbb{F}_q$-subspace of $\mathbb{F}_q^n$. The dual $C^\perp$ of $C$ with respect to $\langle \cdot, \cdot \rangle$ is called {\it the dual code} of $C$. For any $\gamma \in S_n$, we let $C^\gamma := \{ \mathbf{c}^\gamma \mid \mathbf{c} \in C \}$. The {\it permutation automorphism group} of $C$ is
$$
\PAut(C)=\{ \gamma \in S_n \mid C^\gamma=C\}.
$$ 
Clearly, $\PAut(C)=\PAut(C^\perp)$.

From now on, we assume that $C$ is {\it cyclic}, that is, the {\it cyclic shift}  $\sigma=(0,1,\dots,n-1)^{-1}$ is in $\PAut(C)$, or equivalently,
$$
(c_0,\dots,c_{n-1}) \in C \Rightarrow (c_{n-1},c_0,\dots,c_{n-2}) \in C.
$$ In this case, it is well known that $C$ admits a so-called {\it parity-check polynomial} $h(t)=\sum_{i=0}^k a_it^i \in \mathbb{F}_q[t]$, which is a divisor of $t^n-1$. The polynomial $g(t)=\frac{t^n-1}{h(t)}$ is called the {\it generator polynomial} of $C$. Moreover, we have the following description:
\begin{align*}
	C &= \left\{ (c_0,c_1,\ldots,c_{n-1}) \in \mathbb{F}_q^n \mid  a_0 c_{i+k}+a_1c_{i-1+k} + \ldots + a_kc_{i} = 0,\mbox{ for }i \in \{0,1,\ldots,n-k-1\} \right\}.
\end{align*}
Given $\mathbf{c} \in C$, the \emph{minimal polynomial $f=\sum_{i=0}^{\ell-1}b_it^i+t^{\ell} \in \mathbb{F}_q[t]$ of $\mathbf{c}$} is the monic irreducible polynomial with least degree that is annihilated by $\mathbf{c}$, that is,
\begin{align*}
	b_0 c_{i+\ell}+b_1c_{i-1+\ell} + \ldots + b_{\ell-1}c_{i+1} + c_{i} = 0,
\end{align*}
for all $i \in \{0,1,\ldots,p-\ell-1\}$. Note that such a polynomial always exists and divides $h(t)$ since  $h(t)$ is annihilated by $\mathbf{c}$ (see \cite{niederreiter1977weights}).

\section{Quasiprimitive types}\label{sect:quasiprimitive}
In this section, we assume that $G$ is a quasiprimitive group. Clearly, its \emph{socle} $\soc(G)$ (i.e., the subgroup generated by all minimal normal subgroups) is transitive. One can then reduce the study of the intersection density of $G$ by considering its socle, since the latter is transitive. The socles of primitive groups of degree $pq$ were classified by Maru\v{s}i\v{c} and Scapellato in \cite{maruvsivc1994classifying}. In \cite{razafimahatratra2021intersection}, it was proved that the socles of primitive groups have intersection density $1$, except possibly for the ones given in Table~\ref{table:status}.

\begin{table}[H]
	\centering
	\begin{tabular}{|c|c|c|c|}
		\hline
		$\soc(G)$ & $(p,q)$ & action & Information \\
		\hline\hline
		$\po{2d}{2}$ & $\left(2^d - \varepsilon,2^{d-1} + \varepsilon\right)$ & $\begin{aligned}
			\mbox{singular}\\
			\mbox{ 1-spaces}
		\end{aligned}$ & $\begin{aligned}
			&\varepsilon = 1 \mbox{ and $d$ is a Fermat prime} \\ 
			&\varepsilon = -1 \mbox{ and $d-1$ is a Mersenne prime}
		\end{aligned}$\\
		\hline
		$\psl{2}{p}$ & $\left( p,\frac{p+1}{2} \right)$ & cosets of $D_{p-1}$ & $p\geq 13$ and $p\equiv 1 (\operatorname{mod}\ 4)$\\
		\hline
		$\psl{2}{q^2}$ & $\left( \frac{q^2+1}{2},q \right)$ & cosets of $\pgl{2}{q}$ & \\
		\hline
		$\psl{2}{61}$ & $ (61,31)$ & cosets of $\alt(5)$& \\
		\hline
	\end{tabular}
	\caption{The remaining simply primitive groups of degree $pq$.}\label{table:status}
\end{table}

It is in fact conjectured that the groups in Table~\ref{table:status} all have intersection density equal to $1$.
\subsection{A discussion regarding Remark~\ref{thm:quasiprimitive}}
Assume that $G$ is a quasiprimitive group which admits a unique block system $\mathcal{B}$ with blocks of size $q$. Recall that $\Ker(G\to \overline{G})$ is trivial. Consequently, the group $\overline{G}$ acts transitively and faithfully on $\mathcal{B}$. Further, if $\omega\in \Omega$ and $B\in \mathcal{B}$ contains $\omega$, then $G_\omega$ is a subgroup of $\overline{G}_B = G_{\{B\}}$ of index $q$, where $G_{\{B\}}$ is the setwise stabilizer of $B$ in $G$. In other words,
\begin{align}
	[G:G_{\{B\}}] = p \mbox{ and } [G_B:G_\omega] = q.\label{eq:stab-point-block}
\end{align} 

A consequence of the Classification of Finite Simple Groups is that all transitive groups of prime degree can be classified in terms of their socles. Since $G $ is isomorphic to $ \overline{G}$ which is transitive of degree $p$, the socle  of $G$ is one of the groups in Table~\ref{tab:tab1}.
\begin{table}[H]
	\centering
	\begin{tabular}{|c|c|c|}
		\hline
		Line&$\soc(G)$ & Comments \\ \hline 
		1 & $C_p$ & $p$ is prime\\ \hline
		2 & $\alt(p)$ & $p$ is prime \\ \hline
		3 & $\psl{2}{11} $ & degree $11$ (equivalently, acting on cosets of $\alt(5)$)\\ \hline
		4 & $\mathieu{11}$ & of degree $11$ \\ \hline
		5 & $\mathieu{23}$ & of degree $23$ \\ \hline
		6 & $\psl{d}{r^{d^m}}$ & $d\geq 2$ and $p = \frac{r^{d^{m+1}}-1}{r^{d^m}-1}$ where $r$ is a prime \\ \hline
	\end{tabular}
	\caption{Socles of transitive groups of prime degree $p$.}\label{tab:tab1}
\end{table}

We now give an analysis of the intersection density of the groups in Table~\ref{tab:tab1}. 
\begin{itemize}
	\item The group $C_p$ in line~1 cannot be transitive of degree $pq$. Hence, it cannot be quasiprimitive.
	\item For the group $ \psl{2}{11}$ in line~3, the only possibilities for $q$ are $3,5$, and $7$. The point stabilizer $\alt(5)$ of $ \psl{2}{11}$  does not have any subgroup of index in $\{3,7\}$. The alternating group $\alt(4)$ has index $5$ in $\alt(5)$. This gives rise to a transitive action of $\psl{2}{11}$ of degree $55$, whose intersection density is $1$.
	\item Similarly, the Mathieu group $\mathieu{11}$ in line~4 cannot have a quasiprimitive action of degree $q\times 11$, for any $q\in \{3,5,7\}$ since the point stabilizer of $\mathieu{11}$ in its natural action does not admit a subgroup satisfying \eqref{eq:stab-point-block}.
	\item The Mathieu group $\mathieu{23}$ in line~5 admits two transitive actions of degree $11\times 23$, both of which are primitive. There are no other quasiprimitive actions of degree a product of two odd primes.
\end{itemize}
The remaining groups in Table~\ref{tab:tab1} are considered separately.
\begin{lem}
	For any $n\geq 5$ and for any odd prime $q< n$, the group $\alt(n)$ does not admit a transitive action of degree $n\times q$, unless $(n,q) = (5,3)$.
\end{lem}
\begin{proof}
	The point stabilizer of $\alt(n)$ in its natural action on $\{1,2\ldots,n\}$ is $\alt(n-1)$. If $K\leq \alt(n-1)$ such that $[\alt(n-1):K] = q$ is an odd prime, then $\alt(n-1)$ acts transitively on $\alt(n-1)/K$ by left multiplication. If $n\geq 6$, then the group $\alt(n-1)$ is simple, so the kernel of this action of $\alt(n-1)$ is trivial. Consequently, $\alt(n-1)$ embeds into $\sym(\alt(n-1)/K)$ and so $\frac{(n-1)!}{2}\leq q!$, which is impossible. If $n = 5$, the kernel of $\alt(n-1) = \alt(4)$ contains the Klein group $C_2\times C_2$ which is of index $3$ in $\alt(4)$. Hence, $\alt(5)$ admits a transitive action of degree $15$.
\end{proof}
A direct consequence of the above lemma is that the group $\alt(p)$ in line~2 of Table~\ref{tab:tab1} admits a transitive  action of degree $p\times q$ only when $(p,q) = (5,3)$. This action is quasiprimitive and its intersection density can be easily verified to be equal to $1$ via \verb*|Sagemath|.

Therefore, an imprimitive quasiprimitive group of degree $pq$ is an almost simple group containing $\alt(5)$ (transitive of degree $15$), $\psl{2}{11}$ (transitive of degree $55$) or the groups in line~6 of Table~\ref{tab:tab1}.

\subsection{Proof of Theorem~\ref{thm:psl-case}}
Before proving the main result, let us first determine the point stabilizer of $\psl{k}{\ell}$ in its natural action of degree $p$.
Let $p$ be an odd prime such that $p = \frac{\ell^k-1}{\ell-1}$ where $\ell = r^{k^m}$, for some prime number $r$. Let $q$ be an odd prime less than $p$ and assume that $\psl{k}{\ell}$ admits an imprimitive action of degree $pq$, where the projective space $\pg{{k-1}}{\ell}$ of the vector space $\mathbb{F}_\ell^{k}$ is a complete block system. In order for $p= \frac{\ell^k-1}{\ell-1}$ to be a prime, $k$ must be a prime number and $\gcd(k,\ell-1) = 1$. The latter implies that $\psl{k}{\ell} = \sln{k}{\ell}$. A matrix in $\sln{k}{\ell}$ that fixes the $1$-dimensional subspace $\langle e_1\rangle$, where $e_1 \in \mathbb{F}_\ell^k$ is the vector whose first entry is $1$ and $0$ elsewhere, is of the form
\begin{align}
	\begin{bNiceArray}{c|cccc}
		\det(B)^{-1} & a_1 & a_2 & \ldots & a_{k-1}\\
		\hline
		\Block{1-1}<\Large>{\mathbf{0}} &\Block{1-4}<\Huge>{B}\\
	\end{bNiceArray}\label{eq:stab}
\end{align}
where $B\in \operatorname{GL}_{k-1}(\ell)$ and $\left[ a_1\ a_2\ \ldots \ a_{k-1}\right]^T \in \mathbb{F}_{\ell}^{k-1}$.
If $\sln{k}{\ell}$ admits a quasiprimitive action of degree $pq$ for some odd $q<p$, then by \eqref{eq:stab-point-block} the subgroup consisting of all elements of the form \eqref{eq:stab} admits a subgroup of index $q$.

For the remainder of this section, we consider the groups in line~6 of Table~\ref{tab:tab1} for $m = 0$ and $r = k =  q$. Recall that a \emph{Singer subgroup} of $\psl{n}{s}$, for any prime power $s$, is a cyclic subgroup of order $\frac{s^n-1}{s-1}\gcd(n,s-1)$. Let $T$ be a Singer subgroup of $\psl{k}{\ell}$. Since $\gcd(k,\ell-1) = 1$, a Singer subgroup of $\psl{k}{\ell} = \sln{k}{\ell}$  is a regular subgroup of order $p$. Let $\varphi \in \Aut(\mathbb{F}_{\ell^k}/\mathbb{F}_\ell)$ be the Frobenius automorphism of the field $\mathbb{F}_{\ell^k}$. As detailed in \cite[pg.~497]{hestenes1970singer}, an element of $\langle \varphi \rangle$ can be viewed as an invertible matrix over $\mathbb{F}_{\ell^k}$ and the latter induces a collineation of $\pg{k-1}{\ell}$. Let $\langle\psi\rangle$ be the group determined by these collineations. By \cite{hestenes1970singer}, the normalizer of $T$ in $\psl{k}{\ell}$ is given by
\begin{align}
	T\rtimes \langle \psi\rangle \cong C_p \rtimes C_k.
\end{align}
Moreover, $T\rtimes \langle \psi \rangle$ is a Frobenius group by \cite[Theorem~2.10]{hestenes1970singer}. 
By \cite[Lemma~2.6]{maruvsivc1992characterizing} and the fact that $k = q$, the subgroup $T\rtimes \langle \psi\rangle \cong C_p \rtimes C_q$ is a regular subgroup of $\psl{k}{\ell}$. By Lemma~\ref{lem:no-hom}, we conclude that $\rho(\psl{q}{\ell}) = 1$.

\section{Proofs of Theorem~\ref{thm:codes} and Theorem~\ref{thm:characterization}}\label{sect:code-density}
\subsection{Proof of Theorem~\ref{thm:codes}}
Throughout this section, we assume that $G\leq \sym(\Omega)$ is transitive of degree $pq$, where $p>q$ are two odd primes. In addition, we suppose that $\mathcal{B} = \{B_1,B_2,\ldots,B_p\}$ is the unique block system of $G$ and  $H = \Ker(G \to \overline{G})$ is non-trivial. Let $N$ be a minimal normal subgroup of $G$ contained in $H$.
\begin{thm}
	If $N$ is non-abelian, then $\rho(G)=1$.\label{thm:solvable}
\end{thm}
\begin{proof} 
	Since $N\ne 1$ and $q$ is prime, $N$ is transitive on each block in $\mathcal{B}$. By  \cite[Theorem 4.3.A (3)]{dixon1996permutation}, $N$ is a direct product of isomorphic non-abelian simple groups, that is, we can write 
	\begin{align}
		N = S_1\times S_2 \times \ldots \times S_k,\  (k \geq 1),\label{eq:socle} 
	\end{align}
	and there exists a simple group $S$ such that $S_i \cong S$ for all $i\in \{1,2,\ldots,k\}$.

	We identify $H$ (and thus $N$) as a subgroup of $\sym(q)^p$. For $1 \leq j \leq p$, let $N_j$ be the restriction of $N$ to the $j$-th block $B_j$ and denote by $\pi_j:N\twoheadrightarrow N_j \leq \sym(q)$ the canonical restriction. Recall that, by a result due to Burnside, the subgroup $N_j$ is a simple group for any $j \in \{1,2,\ldots,p\}$. If $K$ is any subgroup of $N$, we define 
	$$
	{\rm Supp}(K)=\{ j \in \{1,\dots,p\} \mid {\pi_j}_{|_K} \neq 1\}.
	$$ Moreover, if $K$ is a simple normal subgroup of $N$, then ${\pi_j}_{|_{K}}:K\rightarrow N_j$ is an isomorphism for all $j \in {\rm Supp}(K)$.
	Let $i \in \{ 1,\dots,k\}$. Set 
	$$
	J_i={\rm Supp}(S_i)=\{j \in \{1,\dots,p\} \mid {\pi_j}_{|_{S_i}}: S_i \rightarrow N_j\,\, \textrm{is an isomorphism}\}.
	$$
	\vskip 1mm
	\noindent
	{\bf Fact 1:}  {\it We claim that $J_i \neq \varnothing$, for any $i\in \{1,2,\ldots,k\}$.} 
	
	It is clear that there exists $j \in \{1,\dots, p\}$ such that $\pi_j(S_i) \neq 1$. Since $\pi_j$ is surjective and  $S_i\trianglelefteq N$, we have $\pi_j(S_i)\trianglelefteq N_j$, and so $\pi_j(S_i)=N_j$, as $N_j$ is simple. Now, since ${\rm Ker}({\pi_j}_{|_{S_i}})$ is a proper normal subgroup of $S_i$, we conclude that ${\rm Ker}({\pi_j}_{|_{S_i}})=1$. Therefore $j \in J_i$.
	\vskip 1mm
	\noindent
	{\bf Fact 2:} {\it We claim that}
	\begin{align}
		\bigcup_{i=1}^k J_i=\{1,\dots,p\},\label{eq:union} 
	\end{align} {\it therefore, all $N_j$ are non-abelian.} 

	Let $j\in \{1,2,\ldots,p\}$. By definition, ${\pi_j}_{|_N} = \pi_j \neq 1$, so there exists $i \in \{1,2,\ldots,k\}$ such that ${\pi_{j}}_{|_{S_i}} \neq 1$, or equivalently, $j \in J_i$.
	Therefore, we have $\{1,\dots,p\}={\rm Supp}(N)\subset \bigcup_{i=1}^k J_i$. The other inclusion is obvious.
	\vskip 1mm
	\noindent
	{\bf Fact 3:} {\it We claim that $J_i \cap J_{i'}=\varnothing$ when $i \neq i'$.}
	
	Suppose that $J_i \cap J_{i^\prime} \neq \varnothing$, and pick $j \in J_i \cap J_{i'}$. Since $S_i$ and $S_{i'}$ commute when viewed as subgroups of the direct product in \eqref{eq:socle}, $N_j=\pi_j(S_i)$ and $N_j=\pi_j(S_{i'})$ also commute, and so $N_j$ is abelian, which is a contradiction. Hence,
	\begin{align}
		J_i \cap J_{i'}=\varnothing \mbox{ whenever }i \neq i'.\label{eq:disjoint}
	\end{align}
	\vskip 1mm
	\noindent
	{\bf Fact 4:} {\it We claim that $|J_i|$ does not depend on $i$.} 
	
	Let $i,i' \in \{1,\dots,k\}$. Since $G$ is transitive and $N$ is normal, there exists $g \in G$ such that ${\rm Supp}(gS_i g^{-1}) \cap {\rm Supp}(S_{i'}) \neq \varnothing$. Since $gS_ig^{-1} \trianglelefteq N$ and  $S_{i'}\trianglelefteq N$, either $gS_ig^{-1} \cap S_{i'}=1$ or $gS_ig^{-1} = S_{i'}$. In the first case, $gS_ig^{-1}$ and $S_{i'}$ commute, which implies that, for $j \in {\rm Supp}(gS_i g^{-1}) \cap {\rm Supp}(S_{i'})$, $N_j$ is abelian. Hence, we get a contradiction. Therefore, we must have $gS_ig^{-1} = S_{i'}$. Since $g$ permutes the blocks, the latter implies that
	$$
	|J_i|=|{\rm Supp}(S_i)|=|{\rm Supp}(gS_ig^{-1})| = |{\rm Supp}(S_{i'})|=|J_{i'}|.
	$$ 
	Consequently, there exists a positive integer $m$ such that
	\begin{align}
		\mbox{ $|J_i|$$=m$, for all $i\in \{0,1,\ldots,k\}$.}\label{eq:constant}
	\end{align}
	
	Therefore, by \eqref{eq:union}, \eqref{eq:disjoint}, and \eqref{eq:constant} we have $km=p$. As $p$ is prime, we either have $k=1$ or $k = p$. 
	\vskip 1mm
	\noindent{\bf Fact 5:} {\it $N$ admits a derangement which is a product of $p$ $q$-cycles.} 
	
	Assume that $k=1$. Notice that ${\rm Supp}(N)=\{1,\dots,p\}$, and that, $\pi_j: N \rightarrow N_j$ is an isomorphism, for all $j \in \{1,\dots,p\}$. Since $q$ divides $|N|$, there exists a permutation $\sigma=(\sigma_1,\sigma_2,\dots,\sigma_p)\in N\setminus \{1\}$ of order $q$ (viewed as an element of $\sym(q)^p$). Therefore, $\sigma_j$ is a $q$-cycle or $1$, for any $j\in \{1,2,\ldots,p\}$. However, for $j \in \{1,\dots,p\}$, the map $\pi_j$ is injective, so we have
	$$
	1=\pi_j(1,\dots,1)\neq\pi_j(\sigma)=\sigma_j.
	$$ 
	Hence, all $\sigma_j$ are $q$-cycles, and so $\sigma$ is a product of $p$ $q$-cycles. In other words, $\sigma$ is a semiregular element of $N$.
	
	Next, assume that $k =p$. Since $m = 1$,  each $S_i$ is a subgroup of the $i$-th factor $\sym(q)$. Therefore, 
	$$\{S_1,\dots,S_p\}=\{N_1,\dots, N_p\}.
	$$
	As each $N_i$ is transitive on the block $B_i$, for any $i\in \{1,2,\ldots,p\}$, $N_i$ admits an element of order $q$. From this we deduce that $N = S_1 \times  \ldots \times S_p$ admits an element which is a product of $p$ $q$-cycles, i.e., a semiregular element.
	\vskip 1mm
	\noindent {\bf Fact 6:} {\it The intersection density of $G$ is equal to $1$.}
	
	Since $N$ admits a semiregular element of order $q$ and $N\leq H$, we conclude by Lemma~\ref{lem:semi-regular} that $\rho(G) \leq \rho(\overline{G})$. Since $\overline{G}$ is transitive of prime degree, we have $\rho(\overline{G}) = 1$, which implies $\rho(G) = 1$.
\end{proof}

This proves Theorem~\ref{thm:codes}~\eqref{i}. For the remainder of this section, we consider the case when $N$ is an elementary abelian $q$-group. Before proving the next result, we recall the following lemma, whose proof is straightforward.
\begin{lem}\label{lem:formder}
	Let $\sigma= (0,\dots,q-1)$. The normalizer of $\langle \sigma \rangle$ in $\sym(q)$ is
	$$
	N_{\sym(q)}(\langle \sigma \rangle)=\left\lbrace \tau_{a,i}=
	\begin{pmatrix}
		0 & 1 & \dots & q-1  \\
		a & a+i & \dots & a+(q-1)i
	\end{pmatrix} \mid 0 \leq a \leq q-1, 1 \leq  i \leq q-1
	\right\rbrace.
	$$ Moreover, the derangements in $N_{\sym(q)}(\langle \sigma \rangle)$ are exactly $\sigma,\sigma^2,\dots,\sigma^{q-1}$.
\end{lem}

\begin{thm}
	Assume that $N$ is an elementary abelian $q$-group and $H$ contains a derangement $x$, then $\rho(G)=1$. \label{lem:existence-q-cycles-derangements}
\end{thm}
\begin{proof}
	We claim that $x$ is necessarily a product of $p$ $q$-cycles. First, we identify $H$ as a subgroup of $\sym(q)^p\leq \sym(\Omega)$.
	For $1 \leq i \leq p$, let $\pi_i:\sym(q)^p\twoheadrightarrow \sym(q)$ be the $i$-th projection. Since $1 \neq N\trianglelefteq G$ and $G$ acts transitively on $\Omega$, we get $\pi_i(N)\neq 1 $ ($1 \leq i \leq p$) and $\pi_i(N)$ is an elementary abelian $q$-group of $\sym(q)$. Such a subgroup of $\sym(q)$ has to be cyclic. Therefore, there is a $q$-cycle $\ell_i$ ($1 \leq i \leq p$) such that $\pi_i(N)=\langle \ell_i \rangle$. 
	
	Write $x=(x_1,\dots,x_p) \in \sym(q)^p$, where $x_i$ belongs to the $i$-th factor $\sym(q)$. Now let $1 \leq i \leq p$. Since $N \trianglelefteq H$ and $x \in H$, we get 
	$$\langle \ell_i \rangle=\pi_i(N)=\pi_i(xNx^{-1})=\pi_i(x)\pi_i(N)(\pi_i(x))^{-1}=x_{i}\langle \ell_i \rangle x_{i}^{-1}.
	$$ 
	Therefore, $x_i \in N_{\sym(q)}(\langle \ell_i \rangle)$. Since $x_i$ is a derangement, $x_i$ is a $q$-cycle by Lemma \ref{lem:formder}. Consequently, $x$ is a product of $p$ $q$-cycles.  It is not hard to see that $x$ is a semiregular element such that the set of orbits of $\langle x\rangle$ coincide with $\mathcal{B}$. Therefore, $\rho(G) \leq \rho(\overline{G}) = 1$ by Lemma~\ref{lem:semi-regular}.
\end{proof}
\subsection{Proof of Theorem~\ref{thm:characterization}}\label{sec:conimcod}
First, we recall the construction from \cite{hujdurovic2021intersection}. Let $C$ be a $[p,k]_q$ cyclic code satisfying 
\begin{equation}\label{eq:condzero}
	Z({\bf c})>0\,\,\,\, \textrm{for all}\,\, {\bf c} \in C. 
\end{equation} 
Consider $\Omega=\mathbb{Z}_q \times \mathbb{Z}_p$.
Let $\alpha$ be the permutation of $\Omega$ given by $\alpha: (i,j) \mapsto (i,j+1)$, and for $\mathbf{c} =(c_0,c_1,\ldots,c_{p-1}) \in C$, we let $\beta_{\mathbf{c}}$ be the permutation of $\mathbb{Z}_q \times \mathbb{Z}_p$ such that $\beta_{\mathbf{c}} : (i,j) \mapsto (i+c_j,j)$. Now, define $S = \{ \beta_{\mathbf{c}} \mid \mathbf{c}\in C \}$, $H(C) := \langle S \rangle$, and $G(C) := \langle S \cup \{ \alpha \} \rangle$. It is clear that $G(C)$ is transitive on $\Omega$ and $H(C) \trianglelefteq G(C)$. The orbits of $H(C)$ are of the form  $B_j= \mathbb{Z}_q \times \{ j \}$, for $ 0 \leq j \leq p-1$. We note that for any $\mathbf{c},\mathbf{c}^\prime \in C$, we have
\begin{align}
	\beta_{\mathbf{c}} \beta_{\mathbf{c}^\prime} = \beta_{\mathbf{c}+\mathbf{c}^\prime}.\label{eq:product}
\end{align}
A straightforward computation, using \eqref{eq:product}, proves that 
\begin{align*}
	\Fix(G(C)) = H(C). 
\end{align*}
Moreover, $H(C)$ is derangement-free since $Z(\mathbf{c})>0$, for all $\mathbf{c} \in C$. Hence, $\Gamma_{G(C)}$ is a complete $p$-partite graph and therefore $\rho(G(C)) = q$. Consequently, we have the following result.
\begin{thm}[\cite{hujdurovic2021intersection}]
	If $C$ is a $[p,k]_q$ cyclic code such that $Z(\mathbf{c})>0$, for all $\mathbf{c} \in C$, then there exists an imprimitive group $G(C)$ whose intersection density is equal to $q$.\label{lem:min-density-q}
\end{thm}

Below, we give a proof of Theorem~\ref{thm:characterization}, which in a certain sense, gives a converse to Theorem \ref{lem:min-density-q}.

\begin{thm}\label{thm:converse}
		Let $G \leq {\rm Sym}(\Omega)$ be a degree $pq$ ($p> q$ odd primes) transitive group admitting a unique block system $\mathcal{B}$ whose blocks are of size $q$. Assume that the kernel $H = \Ker(G\to \overline{G})$ is nontrivial and $G$ is minimally transitive with intersection density $\rho(G) = q$. Then, there is a $[p,k]_q$ cyclic code $C$  such that the Hamming weight $w(\mathbf{c})$ is not equal to $p$, for all $\mathbf{c}\in C$.
\end{thm}
\begin{proof}
	Since $H= {\rm Ker}(G \rightarrow \overline{G})$ is nontrivial, there exists a minimal normal subgroup $N$ of $G$ contained in $H$. By Theorem~\ref{thm:codes}, $N$ is an elementary abelian $q$-group and $H$ is derangement-free. Now let $g$ be a derangement of order $p$, which exists by \cite{maruvsivc1981vertex}. Since $\langle N,g \rangle$ is transitive and $G$ is minimally transitive, we deduce that $N=H$.
	Now we are going to construct a $[p,k]_q$ ($k=\log_q(|N|)$) cyclic code $C$ with the property that $Z({\bf c})>0$, for all ${\bf c} \in C$.
	
	Relabel the blocks of $\mathcal{B}$ by $B_0, \dots, B_{p-1}$ in order to obtain $g(B_j)=B_{j+1}$, for $0 \leq j \leq p-1$ and $g(B_{p-1}) = B_0$. Consequently, we may identify $N$ with a subgroup $\langle \sigma_0\rangle \times \dots \times \langle\sigma_{p-1}\rangle$, for some $q$-cycles $\sigma_j \in {\rm Sym}(B_j)$. Moreover, we may assume that $g\sigma_j g^{-1}=\sigma_{j+1}$ for $0 \leq j \leq p-1$. Now write $B_0=\{b_{0,0},\dots,b_{q-1,0}\}$ and suppose that  $\sigma_0=(b_{0,0}\ \dots\ b_{q-1,0})$. For $0 \leq i \leq q-1$ and $0 \leq j \leq p-1$, let $b_{i,j}=b_{i,0}^{g^{j}}$. In this case, we have $B_j=\{b_{0,j},\dots,b_{q-1,j} \}$, for all $0 \leq j \leq p-1$. A direct computation shows that $\sigma_j=(b_{0,j}\ \dots\ b_{q-1,j})$, for all $0 \leq j \leq p-1$. 
	Consider the group isomorphism $\Psi : (a_0,\dots,a_{p-1}) \in \mathbb{F}_q^p=\mathbb{Z}_q^p \mapsto (\sigma_0^{a_0},\dots, \sigma_{p-1}^{a_{p-1}}) \in 
	\langle \sigma_0\rangle \times \dots \times \langle\sigma_{p-1}\rangle$. In this situation $C=\Psi^{-1}(N)$ is a $[p,k]_q$-cyclic code, where $k=\log_q(|N|)$. 
	
	The map $\varphi : b_{i,j} \in \Omega \mapsto (i,j) \in \mathbb{Z}_q \times \mathbb{Z}_p$ is a bijection. By identifying $b_{i,j}$ and $(i,j)$, we deduce an action of $C$ on $\mathbb{Z}_q\times \mathbb{Z}_p$ as follows: 
	$$
			(i,j)^{\Psi(\mathbf{c})} := (i+c_j,j), \mbox{ for }\mathbf{c} = (c_0,c_1,\ldots,c_{p-1}) \in C.
	$$
	For $\mathbf{c} = (c_0,c_1,\ldots,c_{p-1}) \in C$, the element $\Psi(\mathbf{c}) \in N$ admits a fixed point, so there exists $(i,j) \in \mathbb{Z}_q \times \mathbb{Z}_p$ such that $(i,j)^{\Psi(\mathbf{c})} = (i,j)$. Hence, we have $c_j = 0$. It follows that  $Z({\bf c}) >0$ for all ${\bf c} \in C$. 
\end{proof}

\begin{rmk}
The minimal transitivity condition in Theorem \ref{thm:converse} is not necessary. However, one can check that, this condition allows us to recover the structure of $G$ from the constructed code $C$. Indeed, in this case, $G$ is permutation equivalent to the group $G(C)$.
\end{rmk}

\section{Proof of Theorem~\ref{thm:main}}\label{sect:rational}
In this section, we prove Theorem~\ref{thm:main}. Our proof relies on the existence of certain cyclic codes. Now, let $p$ and $q$ be two odd primes such that $p = \frac{q^k-1}{q-1}$ for some prime $k$. Recall that $k$ is the smallest positive integer such that $q^k \congr{1}{p}$. Recall also that $\Phi_p(t)$ is the product of $\frac{\phi(p)}{k} = \frac{p-1}{k}$ monic irreducible polynomials over $\mathbb{F}_q$ of degree $k$. Hence, the splitting field of $\Phi_p(t)$ is $\mathbb{F}_{q^k}$. 

Let $Z \subset \left\{ x\in \mathbb{F}_{q^k} \mid x^p = 1 \right\}$. The \emph{cyclic code with defining zeroes $Z$} is the cyclic code
\begin{align*}
	C_{q,k}(Z) &= \left\{ (c_0,c_1,\ldots,c_{p-1}) \in \mathbb{F}_q^p \mid \sum_{i=0}^{p-1} c_i x^i =0,\ \mbox{ for all } x\in Z \right\}.
\end{align*}
Since $\gcd(p,q)=1$, the map $\varphi:\{ 0,1,\ldots,p-1 \} \to \{0,1,\ldots,p-1\}$ given by: 
\begin{align*}
	i^{\varphi} \congr{qi}{p}
\end{align*} belongs to $S_p$.
Let $\mathbf{c} = (c_0,c_1,\ldots,c_{p-1}) \in C_{q,k}(Z)$ and $x\in Z$. If we denote $\mathbf{c}(x) = c_0 + c_1x + \ldots + c_{p-1}x^{p-1}$, we have 
\begin{align*}
	\mathbf{c}^\varphi(x)=\mathbf{c}(x^{q^{k-1}})=\mathbf{c}(x)^{q^{k-1}}=  0 .
\end{align*}
This is equivalent to saying that $\mathbf{c}^\varphi \in C_{q,k}(Z).$
In other words, $\varphi \in \PAut(C_{q,k}(Z))$.  The permutation automorphism $\varphi$ is called the \emph{Frobenius automorphism }of $C_{q,k}(Z)$. Moreover, since $k$ is the smallest positive integer such that $q^k \congr{1}{p}$, the order of $\varphi$ is $k$.  See \cite{hollmann2008nonstandard} for details.

\subsection{An extension from the Frobenius permutation automorphism}		
Let $\zeta$ be any root of $\Phi_p(t)$ with minimal polynomial $g(t) \in \mathbb{F}_q[t]$. The generator polynomial of $C_{q,k}(\{ \zeta \})$ is clearly $g(t)$, and so $C_{q,k}(\{ \zeta \})$ is a $[p,p-k]_q$ cyclic code. Now, let $C = C_{q,k}(\{ \zeta \})^\perp$ be the dual code of $C_{q,k}(\{ \zeta \})$. Observe that $\PAut(C)=\PAut(C_{q,k}(\{ \zeta \}))$ and $C$ is a $[p,k]_q$ cyclic code.

Next, we consider the group $K = \langle \sigma, \varphi\rangle \leq \PAut(C)$,  where $\sigma$ is the cyclic shift and $\varphi$ is the Frobenius automorphism. By noting that $j^{\varphi^{-1} \sigma \varphi} \congr{j+q}{p}$ for any $j \in \{0,\dots, p-1\}$, we conclude that $\varphi^{-1}\sigma \varphi = \sigma^q$. In other words, $\varphi$ normalizes $\sigma$. Moreover, since $p$ and $q$ are distinct primes, we have $\langle \sigma \rangle \cap \langle \varphi \rangle = \{id\}$. Therefore, $K = \langle \sigma \rangle \rtimes \langle \varphi\rangle \cong C_p \rtimes C_k$.  

Let $\alpha$ and $\beta_\mathbf{c}$ for $\mathbf{c} \in C$ be the permutations of $\Omega_{p,q} = \mathbb{Z}_q \times \mathbb{Z}_p$ defined in the paragraph before Theorem~\ref{lem:min-density-q}. Recall that $H(C)$ is the subgroup generated by $\{ \beta_{\mathbf{c}} \mid \mathbf{c} \in C \}$ and $G(C) = H(C) \rtimes \langle \alpha\rangle$. In addition, we define the permutation $\psi$ of $\Omega_{p,q}$ by $(i,j)^\psi = (i,j^\varphi)$, where $\varphi$ is the Frobenius automorphism of $C$. For any $(i,j)\in \Omega_{p,q}$ and $\mathbf{c} \in C$, we have
\begin{align*}
	\begin{cases}
		(i,j)^{\psi^{-1} \alpha \psi} &= (i,j)^{ \alpha^{q}}\\
		(i,j)^{\psi^{-1} \beta_{{\mathbf{c}}} \psi}&= \left( i+ v_{j^{\varphi^{-1}}},j \right) = (i,j)^{\beta_{\mathbf{c}^{\varphi^{-1}}}}.
	\end{cases}
\end{align*}
As $\varphi \in \PAut(C)$, we conclude that $\psi$ normalizes $G(C)$. In particular, if we let $M(C) := \langle G(C),\psi \rangle$,  noting that $G(C) \cap \langle \psi \rangle=\{id\}$, we have $M(C) = H \rtimes \left( \langle \alpha \rangle \rtimes \langle \psi\rangle \right) \cong \mathbb{F}_q^k \rtimes \left( C_p \rtimes C_k\right)$.
We note that $\psi$ fixes any $(i,0)$ for $i \in \mathbb{Z}_q$, so $\psi$ fixes the block $\mathbb{Z}_q \times \{0\}$ pointwise.

\subsection{The Derangement graph of $M(C)$}\label{sect:der-graph}
In this section, we will study the derangement graph of $M(C)$ defined in the previous section. The group $M:=M(C) = H \rtimes \left( \langle \alpha\rangle \rtimes \langle \psi\rangle \right)$ acts imprimitively on $\Omega_{p,q}$ with block system $\mathcal{B} = \left\{ \mathbb{Z}_q \times \{ j \} \mid j \in \{0,1,\ldots,p-1\} \right\}$. For any $j\in \{0,1,\ldots,p-1\}$, we let $B_j:= \mathbb{Z}_q \times \{ j \}$. We also prove the following theorem.
\begin{thm}
	The intersection density of $M(C)$ is equal to $\max(1,\frac{q}{k})$.\label{thm:density-MC}
\end{thm}

We begin with the following easy lemma.
\begin{lem}\label{lem:eas}
	If $t \in \{1,\dots,k-1\}$, then $q^t-1$ is invertible modulo $p$.
\end{lem}
\begin{proof}
	This follows from the fact that $k$ is the smallest positive integer such that $p \mid q^k-1$.
\end{proof}

\begin{lem}
	For $s \in \{0,\dots,p-1\}$ and $t \in \{0, \dots,k-1\}$ such that $\alpha^s\psi^t\neq 1$, we have
	\begin{align}
		\alpha^s\psi^t \mbox{ fixes a block }\Leftrightarrow \alpha^s\psi^t \mbox{ fixes a unique block } \Leftrightarrow t\neq 0.\label{eq:rmk-fixed-blocks}
	\end{align}
	Moreover, if $\alpha^s\psi^t$ fixes $B\in \mathcal{B}$, then  $\alpha^s \psi^t$ fixes $B$ pointwise. 
\end{lem}	
\begin{proof}	
	For $ j \in \{0,\dots,p-1\}$ we have $B_j^{\alpha^s \psi^t}=B_{q^t(j+s)}$.
	If $t=0$, then the element $\alpha^s\psi^t=\alpha^s$ fixes no block, since $s \neq 0$. Now assume $t \neq 0$. By Lemma \ref{lem:eas}, $q^t-1$ is invertible modulo $p$. Hence, $\alpha^s \psi^t$ fixes a block $B_j$ if and only if $q^t(j+s)=j$, or equivalently, $j\equiv(q^t-1)^{-1} \pmod{p}$. Therefore, we get \eqref{eq:rmk-fixed-blocks}. The last statement is straightforward.
\end{proof}

\subsubsection{Structure of $\Gamma_{M}$}
In this section, we determine the structure of the derangement graph of $M$ using the lexicographic product. Given two graphs $X = (V(X),E(X))$ and $Y = (V(Y),E(Y))$, the \emph{lexicographic product} $X[Y]$ is the graph whose vertex set is $V(X) \times V(Y)$ and  two vertices $(x,y)$ and $(u,v)$ are adjacent in $X[Y]$ if $x \sim_X u$ or $x = u$ and $y \sim_Y v$.

Recall that $\psi$ fixes the elements of the block $\mathbb{Z}_q \times \{0\}$ pointwise. Since $M = H\rtimes \left(\langle \alpha\rangle \rtimes \langle \psi\rangle \right)$, an element of $M$ is of the form $h\alpha^s\psi^t$, for some $h\in H,\ s\in \{0,1,\ldots,p-1\}$ and $t \in \{0,1,\ldots,k-1\}$.

Let us find the structure of the derangement graph of $M$. The point stabilizer $H_{(i,j)}$, for $(i,j)\in \Omega_{p,q}$, is crucial to the structure of this graph. We note that $H$ acts transitively on each element of $\mathcal{B}$. Therefore,  $[H:H_{(i,j)}]=q$.

Now, we decompose the derangement graph of $M$ into cosets of the subgroup $H.$ The cosets $M/H$ partition $M$. 
Therefore, the derangement graph $\Gamma_{M}$ can also be partitioned into $[M:H]$ parts where each part is a coclique of size $|H|$ ($H$ is intersecting). We shall describe the edges of $\Gamma_{M}$ by describing the edges between $H$ and another coset of $H$ in $M$.

We define $H^{(t)}:= \left\{ H \alpha^s \psi^t  \mid  \ 0\leq s\leq p-1  \right\}$, for $0\leq t \leq k-1$. We note that any coset of $H$ in $M$ is in $\bigcup_{t=0}^{k-1}H^{(t)}$ and the latter is a pairwise disjoint union.	

Assume that $ Hy\in H^{(t)}$ and $Hz \in H^{(t^\prime)}$ for $t,t^\prime \in \{0,1,\ldots,k-1\}$. By definition,  for $h_1,h_2\in H$ we have
\begin{align*}
	h_1y \sim h_2z \Leftrightarrow  h_1 \sim h_2zy^{-1}.
\end{align*}
Consequently, we may only consider the adjacency between $H$ and $ Hzy^{-1}$. Assume that $Hzy^{-1} \in H^{(t)}$ for some $t\in \{0,1,\ldots, k-1\}$. That is, there exists $s \in \{0,1,\ldots,p-1\}$ such that $Hzy^{-1} = H \alpha^s\psi^t$. 
Let $\Gamma_{M}[\alpha^s\psi^t]$ be the subgraph of $\Gamma_M$ induced by $H \cup H\alpha^s\psi^t$. 
\begin{itemize}
	\item If $t = 0$, then $B^{\alpha^s\psi^t} \neq B$ for any $B\in \mathcal{B}$ (see \eqref{eq:rmk-fixed-blocks}), unless $s = 0$. If $s\neq 0$ and $t=0$, then for any $(i,j) \in \Omega_{p,q}$, we have
	\begin{align*}
		(i,j)^{n\alpha^s} \neq (i,j)^{n^\prime},
	\end{align*} 
	for any $n,n^\prime \in H$. Therefore, all the possible edges between $H$ and $H\alpha^s \mbox{ occur }$ and so
	\begin{align}
		\Gamma_{M}[\alpha^s] = K_{|H|,|H|} \mbox{ for any }s \in \{1,2,\ldots,p-1\}, \label{eq:multipartite-part}
	\end{align}
	where $K_{|H|,|H|}$ is the complete bipartite graph with parts of equal size $|H|$.
	
	\item By \eqref{eq:rmk-fixed-blocks}, if $t\neq 0$, then there exists a unique $B\in \mathcal{B}$ such that $B^{\alpha^s\psi^t} = B$, and moreover, $\alpha^s \psi^t$ fixes $B$ pointwise. Assume without loss of generality that this block $B$ is $B_\ell = \mathbb{Z}_q \times \{\ell\}$, for some $\ell \in \mathbb{Z}_p $. Let $\beta_{\mathbf{c}}\in H,$ where $\mathbf{c} = (c_0,c_1,\ldots,c_{p-1}) \in C$ and $c_\ell \neq 0$  (i.e., $\langle\beta_\mathbf{c}\rangle$ is transitive on $B$). The subgroup $\langle \beta_{{\mathbf{c}}}\rangle = \{ id,\beta_{{\mathbf{c}}},\beta_{\mathbf{c}}^2,\ldots,\beta_{\mathbf{c}}^{q-1} \}$ is a right transversal of $H_{(0,\ell)}$ in $H$, that is, $\left|\langle \beta_{\mathbf{c}}\rangle \cap H_{(0,\ell)} y\right| = 1$ for any $y\in H$.
	Now we will see that we can precisely determine the edges of $\Gamma_M[\alpha^s\psi^t]$.
	
	The subgroup $H$ is partitioned by the cosets $\left(H_{(0,\ell)} \beta_{\mathbf{c}}^l\right)_{l=0,1,\ldots,q-1}$. 
	Now, we determine the edges between $H_{(0,\ell)} \beta_{\mathbf{c}}^l$ and $H_{(0,\ell)} \beta_{\mathbf{c}}^{l^\prime} \alpha^s \psi^t$, for $l,l^\prime \in \{0,1,\ldots,q-1\}$. Assume that $(i^\prime,j^\prime)\in B^\prime = B_{\ell^\prime}$, where $\ell^\prime \neq \ell$. Since $\alpha^s\psi^t$ only fixes $B$ and $H_{(0,\ell)} \beta_{\mathbf{c}}^l \subset H$ which is the kernel of the action on $\mathcal{B}$, we know that  for any $l,l^\prime\in \{0,1,\ldots,q-1\}$
	\begin{align*}
		\left(i^\prime,j^\prime\right)^{H_{(0,\ell)} \beta_{\mathbf{c}}^{l^\prime} \alpha^s \psi^t} \cap B^\prime = \varnothing \mbox{ and } \left(i^\prime,j^\prime\right)^{H_{(0,\ell)}  \beta_{\mathbf{c}}^l} \subset B^\prime .
	\end{align*}
	In other words, no elements of $H_{(0,\ell)} \beta_{\mathbf{c}}^l$ can intersect an element of $H_{(0,\ell)} \beta_{\mathbf{c}}^{l^\prime} \alpha^s\psi^t$ outside of $B$. Let us examine the possible elements of $B$ on which two permutations from the two cosets can intersect.  Since $H_{(0,\ell)}$ fixes $B$ pointwise, if $l = l^\prime$ and $(i,\ell) = (0,\ell)^{\beta_{\mathbf{c}}^l}$, then for any $z,z^\prime\in H_{(0,\ell)} = H_{(i,\ell)} $ we have
	\begin{align*}
		(0,\ell)^{z^\prime \beta_{\mathbf{c}}^{l^\prime} \alpha^s \psi^t} &= \left((0,\ell)^{z^\prime \beta_{\mathbf{c}}^{l^\prime}}\right)^{\alpha^s\psi^t} = (0,\ell)^{z^\prime \beta_{\mathbf{c}}^{l^\prime}} = (0,\ell)^{\beta_{\mathbf{c}}^{l^\prime}} = (0,\ell)^{\beta_{\mathbf{c}}^l} = (i,\ell)\\
		(0,\ell)^{z\beta_{\mathbf{c}}^l} &= (0,\ell)^{\beta_{\mathbf{c}}^l} = (i,\ell).
	\end{align*}
	In other words, any permutation in $H_{(0,\ell)} \beta_{\mathbf{c}}^{l^\prime} \alpha^s \psi^t$ intersects any permutation in $H_{(0,\ell)} \beta_{\mathbf{c}}^l$. 
	
	If $l\neq l^\prime$, then $\beta_{\mathbf{c}}^l\beta_{\mathbf{c}}^{-l^\prime}$ does not have a fixed point in $B$ since $\langle \beta_{\mathbf{c}}\rangle$ is a right transversal of $H_{(0,\ell)}$ in $H$. For any $z,z^\prime \in H_{(0,\ell)}$, we have
	\begin{align*}
		(i,\ell)^{z\beta_{\mathbf{c}}^l} &= {(i,\ell)}^{\beta_{\mathbf{c}}^l}\\
		(i,\ell)^{z^\prime \beta_{\mathbf{c}}^{l^\prime} \alpha^s\psi^t} &= (i,\ell)^{z^\prime \beta_{\mathbf{c}}^{l^\prime}} = (i,\ell)^{\beta_{\mathbf{c}}^{l^\prime}} \neq (i,\ell)^{\beta_{\mathbf{c}}^l}.
	\end{align*}
	Consequently, no permutation in $H_{(i,j)} \beta_{\mathbf{c}}^l$ can intersect any permutation of $H_{(i,j)} \beta_{\mathbf{c}}^{l^\prime} \alpha^s \psi^t$, for $l\neq l^\prime$.
	In summary, for any $t\neq 0$ the graph induced by 
	\begin{align}
		\begin{cases}
			H_{(0,\ell)} \beta_{\mathbf{c}}^l \cup H_{(0,\ell)} \beta_{\mathbf{c}}^{l^\prime} \alpha^s \psi^t \mbox{ is a coclique} & \mbox{ if } l = l^\prime\\
			H_{(0,\ell)} \beta_{\mathbf{c}}^l \cup H_{(0,\ell)} \beta_{\mathbf{c}}^{l^\prime} \alpha^s \psi^t \mbox{ is complete bipartite} & \mbox{ if } l \neq l^\prime . \label{eq:non-multipartite-part}
		\end{cases}
	\end{align}
\end{itemize}

Let $K_q^2$ be the complete bipartite graph $K_{q,q}$. In addition, let $\widetilde{K}_{q}^2$ be the graph obtained from $K_q^2$ by removing $q$ vertex-disjoint edges, i.e., a perfect matching. Combining \eqref{eq:multipartite-part} and \eqref{eq:non-multipartite-part}, we conclude that 
\begin{align*}
	\Gamma_{M}[\alpha^s\psi^t]
	=
	\begin{cases}
		K_{q}^2[\Gamma_{H_{(0,\ell)}}] &\mbox{ if } t=0\\
		\widetilde{K}_{q}^2[\Gamma_{H_{(0,\ell)}}] & \mbox{ otherwise}.
	\end{cases}
\end{align*}

From this, the adjacency in $\Gamma_M$ between $H \alpha^s \psi^t$ and $H\alpha^{s^\prime} \psi^{t^\prime}$ is completely determined by the values of $t-t^\prime$. If $t = t^\prime$, then the subgraph induced by $H \alpha^s \psi^t \cup H \alpha^{s^\prime} \psi^{t^\prime}$, where $s\neq s^\prime$, in $\Gamma_M$ is isomorphic to $\Gamma_M[\alpha^{s-s'}]$ and therefore it is  isomorphic to $K_{q}^2[\Gamma_{H_{(0,\ell)}}]$. If $t\neq t^\prime$, then the subgraph induced by $H \alpha^s \psi^t \cup H \alpha^{s^\prime} \psi^{t^\prime}$ in $\Gamma_M$ is isomorphic to  $\Gamma_M[\alpha^{s-q^{t'-t}s'}\psi^{t'-t}]$, and so it is isomorphic to $\widetilde{K}_{q}^2[\Gamma_{H_{(0,\ell)}}]$.

\subsubsection{Proof of Theorem \ref{thm:density-MC}}
Let $\mathcal{F}$ be a coclique of $\Gamma_M.$ Without loss of generality, we may assume that $id\in \mathcal{F}$ since we can always multiply $\mathcal{F}$ by the inverse of one of its elements and obtain another intersecting set. 

First assume that $\mathcal{F}$ is not contained in $H$. For any $t\in \{0,1,\ldots,k-1\}$, we define $\mathcal{F}_t = \mathcal{F} \cap \left( \bigcup_{s=0}^{p-1}H \alpha^s \psi^t \right)$. 
\vskip 1mm
\noindent
{\bf Claim 1:} {\it For any $t\in \{1,\ldots,k-1\}$, we have $|\mathcal{F}_t|\leq \frac{|H|}{q}$.} 

Indeed, from the structure of $\Gamma_M$, $\mathcal{F}_t$ is contained in $H \alpha^s\psi^t$ for some $s \in \{0,\dots,p-1\}$. Now let $g=\beta_{\mathbf{c}}\alpha^s\psi^t \in \mathcal{F}_t$. Since $g \not\sim_{\Gamma_{M}} id$, there exists $(i,j) \in \Omega_{p,q}$ such that 
$(i,j)=(i,j)^g$, that is, $(i,j)=(i+c_j,q^t(j+s))$. From Lemma \ref{lem:eas}, $j=-(q^t-1)^{-1}s$ and $c_j=0$. Therefore, there are at most $\frac{|H|}{q}$ possible choices for $\mathbf{c}$, and so $|\mathcal{F}_t| \leq  \frac{|H|}{q}$.
\vskip 1mm
\noindent
{\bf Claim 2:} {\it We have $|\mathcal{F}_0|\leq \frac{|H|}{q}$.} 

Indeed, $\mathcal{F}_0$ is contained in $H\alpha^s$ for some  $s \in \{0,\dots,p-1\}$. The identity element $id$ is non-adjacent to every element of $\mathcal{F}_t$ and no element in $H\alpha^s$ has a fixed point for $s \in \{1,\dots,p-1\}$, hence $s=0$ and $\mathcal{F}_0 \subset H$. Consider $ u \in \mathcal{F} \setminus H \neq \varnothing$ and write $u^{-1}=\beta_{\mathbf{c'}}\alpha^{s'}\psi^{t'}$. Now let $h=\beta_{\mathbf{c}} \in \mathcal{F}_0$. Since $h \not\sim u$, there exists $(i,j) \in \Omega_{p,q}$ such that $(i,j)=(i,j)^{hu^{-1}}=(i,j)^{\beta_{\mathbf{c}+\mathbf{c'}}\alpha^{s'}\psi^{t'}}$. Note that $t' \neq 0$ since $\beta_{\mathbf{c}+\mathbf{c'}}\alpha^{s'}$ fixes no point when $s' \neq 0$. The same argument as in Claim 1 proves that there are at most $\frac{|H|}{q}$ choices for $\mathbf{c}+\mathbf{c'}$ and thus for $\mathbf{c}$. Therefore, $|\mathcal{F}_0| \leq \frac{|H|}{q}$.
\vskip 1mm
\noindent
We conclude that
\begin{align*}
	|\mathcal{F}| = \sum_{t=0}^{k-1}|\mathcal{F}_t|\leq \frac{k|H|}{q}.
\end{align*}
Moreover, this bound is sharp since it is attained by the coclique $\bigcup_{t=0}^{k-1} H_{(1,0)} \psi^t$. Hence, we have
\begin{align*}
	\rho(\mathcal{F}) \leq \frac{\frac{k |H|}{q}}{\frac{|M|}{pq}} = 1 .
\end{align*}

Finally, assume that $\mathcal{F}$ is contained in $H$. Then, it is clear that $|\mathcal{F}| \leq |H|$ and so
\begin{align*}
	\rho(\mathcal{F}) \leq \frac{|H|}{\frac{|M|}{pq}} = \frac{|H|pq}{|H|pk} = \frac{q}{k}.
\end{align*}
The latter is clearly attained by $H$. Consequently, we have $\rho(M) = \max(1,\frac{q}{k})$. Moreover, a maximum coclique of $\Gamma_M$ is either a coset of $H$ or is of the form
\begin{align*}
	\bigcup_{t = 0}^{k-1} H_{(0,0)}\beta_{\mathbf{c}}^{l_t}\alpha^{s_t} \psi^t,
\end{align*}
for the sequences $(l_t)_{t = 0,1,\ldots,k-1} \subset \{ 0,1,\ldots,q-1 \}$ and $(s_t)_{t=0,1,\ldots,k-1}\subset \{0,1,\ldots,p-1\}$.

This completes the proof of Theorem~\ref{thm:density-MC}.

\subsection{Proof of Theorem~\ref{thm:main}}
	The group $M(C)$  acts imprimitively on $\Omega_{p,q}$ with block system $\mathcal{B}=\{\mathbb{Z}_q \times \{j\} \mid j \in \{0,1,\dots,p-1\}\}$. By Theorem~\ref{thm:density-MC}, $\rho(M(C)) = \max( 1,\frac{q}{k} )$. In particular, if $p = \frac{q^k-1}{q-1}$ such that $k<q$, then $\rho(M(C)) = \frac{q}{k}$ is a rational number.

\section{Future work}\label{sect:future}
In this paper, we studied the intersection density of imprimitive groups of degree a product of two odd primes $p>q$. For transitive groups of degree $pq$ admitting a unique block system with blocks of size $q$, as well as non-trivial normal subgroups, we proved Theorem~\ref{thm:codes} and Theorem~\ref{thm:main}. In addition, we proved that if  $G\leq \sym(\Omega)$ is imprimitive and quasiprimitive of degree $pq$, then $G$ is an almost simple group containing $\alt(5)$ or a projective special linear group. The former has intersection density $1$ and we conjectured in Conjecture~\ref{conj2} that the latter also has intersection density $1$. If Conjecture~\ref{conj2} is true, then we will obtain that $\mathcal{I}_{pq} = \{1\}$ if and only if there exists no $[p,k]_q$ cyclic code $C$ such that $Z(\mathbf{c})>0$, for any $\mathbf{c} \in C$. It is likely that Conjecture~\ref{conj2} is too hard for the proof techniques that are available in the literature.

Combining  \cite[Conjecture~9.1]{razafimahatratra2021intersection}, Conjecture~\ref{conj1}, and Conjecture~\ref{conj2}, we make the following conjecture.
\begin{conj}
	For any two odd primes $p>q$, we have
	\begin{enumerate}[(i)]
		\item $\mathcal{I}_{pq} = \{1\}$ if and only if there exists no $[p,k]_q$ cyclic codes $C$ such that $Z(\mathbf{c})>0$, for any $\mathbf{c} \in C$.
		\item If $C$ is a $[p,k]_q$ cyclic code such that $Z(\mathbf{c})>0$, for $\mathbf{c} \in C$, and $\sigma \in \PAut(C)$ is the cyclic shift automorphism, then 
		\begin{align*}
			\mathcal{I}_{pq} = \left\{\tfrac{q}{k} \mid \operatorname{N}_{\PAut(C)}(\langle \sigma \rangle) \mbox{ admits an element of order $k<q$} \right\} \cup \{1\}.
		\end{align*}
	\end{enumerate}\label{conj:main}
\end{conj}

\vskip 5mm
\noindent{\bf Acknowledgement.} \ 
We would like to thank the anonymous referees for their comments and suggestions regarding the content and the presentation of the paper.

The first author is grateful for the support of the Israel Science Foundation (grant no. 353/21).
This work was completed when the second and third authors were Ph.D students at the Department of Mathematics and Statistics, University of Regina, Canada.

 \end{document}